\documentclass[reqno,final,12pt]{amsart}
\usepackage{amsmath,amssymb,amsthm,amscd,verbatim}
\usepackage{color}
\usepackage{lscape}
\usepackage{fullpage}
\usepackage{stmaryrd}
\usepackage[notref,notcite,color]{showkeys}
\usepackage{epic,eepic}
\usepackage[dvips]{graphicx}
\usepackage[T1]{fontenc}

\numberwithin{equation}{section}
\usepackage{textcomp}
\usepackage{amsbsy}

\newtheorem{theorem}{Theorem}[section]
\newtheorem{definition}[theorem]{Definition}

\theoremstyle{remark}

\newcommand{\Hom}{\mathrm{Hom}}

\newcommand{\eval}[2]{\llbracket #1 \rrbracket_{#2}}
\newcommand{\df}{\stackrel{\rm def}{=}}

\def\veca{\mathbf}
\def\vecb{\mathbf}

\begin{document}

\title{Non-three-colorable common graphs exist}

\author{Hamed Hatami}
\address{School of Computer Science, McGill University, Montreal, Canada.}
\email{hatami@cs.mcgill.ca}
\author{Jan Hladk\'y}
  \address{Department of Applied Mathematics,
        Faculty of Mathematics and Physics, Charles University,
        Malostransk\'e n\'am\v{e}st\'{\i}~25, 118~00 Prague,
        Czech Republic, and DIMAP, Department of Computer
       Science, University of Warwick, Coventry CV4 7AL,
United Kingdom,}
  \email{honzahladky@gmail.com}
\author{Daniel Kr\'al'}
\address{Institute for Theoretical Computer Science, Faculty of Mathematics and Physics, Charles University,
        Malostransk\'e n\'am\v{e}st\'{\i}~25, 118~00 Prague, Czech Republic.}
\email{kral@kam.mff.cuni.cz}
\author{Serguei Norine}
\address{Department of Mathematics, Princeton University, Princeton, NJ, USA.}
\email{snorin@math.princeton.edu}
\author{Alexander Razborov}
\address{Department of Computer Science, University of Chicago, IL, USA.}
\email{razborov@cs.uchicago.edu}
\thanks{HH was partially supported by an NSERC Discovery Grant.
JH was supported by EPSRC award EP/D063191/1.
DK:~The~work leading to this invention has received funding from the European Research Council under the European Union's Seventh Framework Programme (FP7/2007-2013)/ERC grant agreement no.~259385. SN was supported in part by NSF
  under Grant No.~DMS-0803214. Part of AR's work was done while
the author was at Steklov Mathematical Institute, supported by the Russian Foundation for Basic Research, and at Toyota Technological Institute, Chicago.}

\begin{abstract}
A graph $H$ is called \emph{common} if the total number of copies of $H$ in every graph and its complement asymptotically minimizes for random graphs. A former conjecture of Burr and Rosta, extending a conjecture of Erd\H{o}s asserted that every graph is common. Thomason
%~\cite{MR991659}
 disproved both conjectures by showing that $K_4$ is not common. It is now known that in fact the common graphs are very rare. Answering a question of Sidorenko
 %~\cite{MR1605401}
 and of Jagger, {\v{S}}{\v{t}}ov{\'{\i}}{\v{c}}ek and Thomason
 %~\cite{MR1394515}
 from~1996 we show that the $5$-wheel is common. This provides the first example of a common graph that is not three-colorable.
\end{abstract}

\maketitle

\section{Introduction \label{sec:intro}}

A natural question in extremal graph theory is how many monochromatic subgraphs isomorphic to a graph $H$ must be contained in any two-coloring
of the edges of the complete graph $K_n$. Equivalently, how many subgraphs isomorphic to a graph $H$ must be contained in a graph and its complement?

Goodman~\cite{MR0107610} showed that for $H=K_3$, the optimum solution is essentially obtained by a typical random graph. The graphs $H$ that satisfy this property are called \emph{common}. Erd\H{o}s~\cite{MR0151956} conjectured that all complete graphs are common. Later, this conjecture was extended to all graphs by Burr and Rosta~\cite{MR595601}. Sidorenko~\cite{MR1033422} disproved Burr and Rosta's conjecture by showing that a triangle with a pendant edge is not common. Later  Thomason~\cite{MR991659}  disproved Erd\H{o}s's conjecture by showing that for $p \ge 4$, the complete graphs $K_p$ are not common.  It is now known that in fact the common graphs are very rare.
For example, Jagger, \v{S}\v{t}ov\'{\i}\v{c}ek and Thomason~\cite{MR1394515} showed that every graph that contains $K_4$ as a subgraph is not common. If we work with $k$-edge-colorings of $K_n$ rather than 2-edge-colorings we get the notion of a \emph{$k$-common} graph.  Cummings and Young~\cite{CummYoung} recently proved that no graph containing the triangle $K_3$ is 3-common, a counterpart of the result of Jagger, \v{S}\v{t}ov\'{\i}\v{c}ek and Thomason above.

There are some classes of graphs that are known to be common. Sidorenko~\cite{MR1033422} showed that cycles are common. A conjecture due to Erd\H{o}s and Simonovits~\cite{MR776802} and Sidorenko \cite{MR1138091,MR1225933} asserts that for every bipartite graph $H$, among graphs of given density random graphs essentially contain the least number of subgraphs isomorphic to $H$.
It is not hard to see that every graph $H$ with the latter property is common, therefore
this conjecture would imply that all bipartite graphs are common. The Erd\H{o}s-Simonovits-Sidorenko conjecture has been verified for a handful of graphs~\cite{MR1225933,MR1605401,HH,CFS}, and hence there are various classes of bipartite graphs that are known to be common. In~\cite{MR1394515} and~\cite{MR1605401} some graph operations are introduced that can be used to ``glue'' common graphs in order to construct new common graphs. However none of these operations can increase the chromatic number to a number larger than three, and as a result, all of the known common graphs are of chromatic number at most $3$. With these considerations Jagger, \v{S}\v{t}ov\'{\i}\v{c}ek and Thomason~\cite{MR1394515} state ``We regard the determination of the commonality of $W_5$ [the wheel with $5$ spokes] as the most interesting open problem in the area.''

We will prove in Theorem~\ref{thm:main} that $W_5$ (see Figure~\ref{fig:wheel}) is common.
\begin{figure}[ht]
\begin{center}
% This file was made with: LaTeXPiX  (Build 3618)
% Coded by: N.J.H.M. van Beurden
% Email: beurden@email.com
% Webpage: http://www.beurden.cjb.net
\setlength{\unitlength}{0.254mm}
\begin{picture}(64,83)(28,-96)
        \special{color rgb 0 0 0}\allinethickness{0.254mm}\special{sh 0.99}\put(40,-75){\ellipse{4}{4}} % Shade Dot
        \special{color rgb 0 0 0}\allinethickness{0.254mm}\special{sh 0.99}\put(80,-75){\ellipse{4}{4}} % Shade Dot
        \special{color rgb 0 0 0}\allinethickness{0.254mm}\special{sh 0.99}\put(90,-40){\ellipse{4}{4}} % Shade Dot
        \special{color rgb 0 0 0}\allinethickness{0.254mm}\special{sh 0.99}\put(30,-40){\ellipse{4}{4}} % Shade Dot
        \special{color rgb 0 0 0}\allinethickness{0.254mm}\special{sh 0.99}\put(60,-15){\ellipse{4}{4}} % Shade Dot
        \special{color rgb 0 0 0}\allinethickness{0.254mm}\path(30,-40)(60,-15) % Plain Solid Line
        \special{color rgb 0 0 0}\allinethickness{0.254mm}\path(60,-15)(90,-40) % Plain Solid Line
        \special{color rgb 0 0 0}\allinethickness{0.254mm}\special{sh 0.99}\put(40,-75){\ellipse{4}{4}} % Shade Dot
        \special{color rgb 0 0 0}\allinethickness{0.254mm}\special{sh 0.99}\put(80,-75){\ellipse{4}{4}} % Shade Dot
        \special{color rgb 0 0 0}\allinethickness{0.254mm}\special{sh 0.99}\put(90,-40){\ellipse{4}{4}} % Shade Dot
        \special{color rgb 0 0 0}\allinethickness{0.254mm}\special{sh 0.99}\put(30,-40){\ellipse{4}{4}} % Shade Dot
        \special{color rgb 0 0 0}\allinethickness{0.254mm}\special{sh 0.99}\put(60,-15){\ellipse{4}{4}} % Shade Dot
        \special{color rgb 0 0 0}\allinethickness{0.254mm}\path(30,-40)(60,-15) % Plain Solid Line
        \special{color rgb 0 0 0}\allinethickness{0.254mm}\path(60,-15)(90,-40) % Plain Solid Line
        \special{color rgb 0 0 0}\allinethickness{0.254mm}\special{sh 0.99}\put(40,-75){\ellipse{4}{4}} % Shade Dot
        \special{color rgb 0 0 0}\allinethickness{0.254mm}\special{sh 0.99}\put(80,-75){\ellipse{4}{4}} % Shade Dot
        \special{color rgb 0 0 0}\allinethickness{0.254mm}\special{sh 0.99}\put(90,-40){\ellipse{4}{4}} % Shade Dot
        \special{color rgb 0 0 0}\allinethickness{0.254mm}\special{sh 0.99}\put(30,-40){\ellipse{4}{4}} % Shade Dot
        \special{color rgb 0 0 0}\allinethickness{0.254mm}\special{sh 0.99}\put(60,-15){\ellipse{4}{4}} % Shade Dot
        \special{color rgb 0 0 0}\allinethickness{0.254mm}\path(30,-40)(60,-15) % Plain Solid Line
        \special{color rgb 0 0 0}\allinethickness{0.254mm}\path(60,-15)(90,-40) % Plain Solid Line
        \special{color rgb 0 0 0}\allinethickness{0.254mm}\special{sh 0.99}\put(40,-75){\ellipse{4}{4}} % Shade Dot
        \special{color rgb 0 0 0}\allinethickness{0.254mm}\special{sh 0.99}\put(80,-75){\ellipse{4}{4}} % Shade Dot
        \special{color rgb 0 0 0}\allinethickness{0.254mm}\special{sh 0.99}\put(90,-40){\ellipse{4}{4}} % Shade Dot
        \special{color rgb 0 0 0}\allinethickness{0.254mm}\special{sh 0.99}\put(30,-40){\ellipse{4}{4}} % Shade Dot
        \special{color rgb 0 0 0}\allinethickness{0.254mm}\special{sh 0.99}\put(60,-15){\ellipse{4}{4}} % Shade Dot
        \special{color rgb 0 0 0}\allinethickness{0.254mm}\path(30,-40)(60,-15) % Plain Solid Line
        \special{color rgb 0 0 0}\allinethickness{0.254mm}\path(60,-15)(90,-40) % Plain Solid Line
        \special{color rgb 0 0 0}\allinethickness{0.254mm}\path(40,-75)(30,-40) % Plain Solid Line
        \special{color rgb 0 0 0}\allinethickness{0.254mm}\path(80,-75)(90,-40) % Plain Solid Line
        \special{color rgb 0 0 0}\allinethickness{0.254mm}\special{sh 0.99}\put(40,-75){\ellipse{4}{4}} % Shade Dot
        \special{color rgb 0 0 0}\allinethickness{0.254mm}\special{sh 0.99}\put(80,-75){\ellipse{4}{4}} % Shade Dot
        \special{color rgb 0 0 0}\allinethickness{0.254mm}\special{sh 0.99}\put(90,-40){\ellipse{4}{4}} % Shade Dot
        \special{color rgb 0 0 0}\allinethickness{0.254mm}\special{sh 0.99}\put(30,-40){\ellipse{4}{4}} % Shade Dot
        \special{color rgb 0 0 0}\allinethickness{0.254mm}\special{sh 0.99}\put(60,-15){\ellipse{4}{4}} % Shade Dot
        \special{color rgb 0 0 0}\allinethickness{0.254mm}\path(30,-40)(60,-15) % Plain Solid Line
        \special{color rgb 0 0 0}\allinethickness{0.254mm}\path(60,-15)(90,-40) % Plain Solid Line
        \special{color rgb 0 0 0}\allinethickness{0.254mm}\special{sh 0.99}\put(40,-75){\ellipse{4}{4}} % Shade Dot
        \special{color rgb 0 0 0}\allinethickness{0.254mm}\special{sh 0.99}\put(80,-75){\ellipse{4}{4}} % Shade Dot
        \special{color rgb 0 0 0}\allinethickness{0.254mm}\special{sh 0.99}\put(90,-40){\ellipse{4}{4}} % Shade Dot
        \special{color rgb 0 0 0}\allinethickness{0.254mm}\special{sh 0.99}\put(30,-40){\ellipse{4}{4}} % Shade Dot
        \special{color rgb 0 0 0}\allinethickness{0.254mm}\special{sh 0.99}\put(60,-15){\ellipse{4}{4}} % Shade Dot
        \special{color rgb 0 0 0}\allinethickness{0.254mm}\path(30,-40)(60,-15) % Plain Solid Line
        \special{color rgb 0 0 0}\allinethickness{0.254mm}\path(60,-15)(90,-40) % Plain Solid Line
        \special{color rgb 0 0 0}\allinethickness{0.254mm}\special{sh 0.99}\put(40,-75){\ellipse{4}{4}} % Shade Dot
        \special{color rgb 0 0 0}\allinethickness{0.254mm}\special{sh 0.99}\put(80,-75){\ellipse{4}{4}} % Shade Dot
        \special{color rgb 0 0 0}\allinethickness{0.254mm}\special{sh 0.99}\put(90,-40){\ellipse{4}{4}} % Shade Dot
        \special{color rgb 0 0 0}\allinethickness{0.254mm}\special{sh 0.99}\put(30,-40){\ellipse{4}{4}} % Shade Dot
        \special{color rgb 0 0 0}\allinethickness{0.254mm}\special{sh 0.99}\put(60,-15){\ellipse{4}{4}} % Shade Dot
        \special{color rgb 0 0 0}\allinethickness{0.254mm}\path(30,-40)(60,-15) % Plain Solid Line
        \special{color rgb 0 0 0}\allinethickness{0.254mm}\path(60,-15)(90,-40) % Plain Solid Line
        \special{color rgb 0 0 0}\allinethickness{0.254mm}\special{sh 0.99}\put(40,-75){\ellipse{4}{4}} % Shade Dot
        \special{color rgb 0 0 0}\allinethickness{0.254mm}\special{sh 0.99}\put(80,-75){\ellipse{4}{4}} % Shade Dot
        \special{color rgb 0 0 0}\allinethickness{0.254mm}\special{sh 0.99}\put(90,-40){\ellipse{4}{4}} % Shade Dot
        \special{color rgb 0 0 0}\allinethickness{0.254mm}\special{sh 0.99}\put(30,-40){\ellipse{4}{4}} % Shade Dot
        \special{color rgb 0 0 0}\allinethickness{0.254mm}\special{sh 0.99}\put(60,-15){\ellipse{4}{4}} % Shade Dot
        \special{color rgb 0 0 0}\allinethickness{0.254mm}\path(30,-40)(60,-15) % Plain Solid Line
        \special{color rgb 0 0 0}\allinethickness{0.254mm}\path(60,-15)(90,-40) % Plain Solid Line
        \special{color rgb 0 0 0}\put(50,-96){\shortstack{$W_5$}} % Plain Text
        \special{color rgb 0 0 0}\allinethickness{0.254mm}\path(40,-75)(30,-40) % Plain Solid Line
        \special{color rgb 0 0 0}\allinethickness{0.254mm}\path(80,-75)(90,-40) % Plain Solid Line
        \special{color rgb 0 0 0}\allinethickness{0.254mm}\path(40,-75)(80,-75) % Plain Solid Line
        \special{color rgb 0 0 0}\allinethickness{0.254mm}\special{sh 0.3}\put(60,-50){\ellipse{4}{4}} % Shade Dot
        \special{color rgb 0 0 0}\allinethickness{0.254mm}\path(60,-50)(60,-15) % Plain Solid Line
        \special{color rgb 0 0 0}\allinethickness{0.254mm}\path(60,-50)(90,-40) % Plain Solid Line
        \special{color rgb 0 0 0}\allinethickness{0.254mm}\path(60,-50)(30,-40) % Plain Solid Line
        \special{color rgb 0 0 0}\allinethickness{0.254mm}\path(60,-50)(40,-75) % Plain Solid Line
        \special{color rgb 0 0 0}\allinethickness{0.254mm}\path(60,-50)(80,-75) % Plain Solid Line
        \special{color rgb 0 0 0} % Set color to black again (default font color)
\end{picture}
\caption{\label{fig:wheel} The 5-wheel.}
\end{center}
\end{figure}
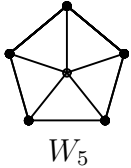
 This will also answer a question of Sidorenko~\cite{MR1605401}. He showed~\cite[Theorem~8]{MR1605401} that every graph that is obtained by adding a vertex of full degree to a bipartite graph of average degree at least one satisfying the Erd\H{o}s-Simonovits-Sidorenko conjecture is common. Sidorenko further asked whether in this theorem both conditions of being bipartite and having average degree at least one are essential in order to obtain a common graph. Our result answers his question in the negative, as $W_5$ is obtained by adding a vertex of full degree to a non-bipartite graph.

The proof of Theorem~\ref{thm:main} is a rather standard Cauchy-Schwarz calculation in flag algebras~\cite{MR2371204}, and is generated with the aid of a computer using semi-definite programming.
A similar approach was successfully applied for example in~\cite{RazborovTuran,HKN,Talbot,Grzesik,HHKNR}.
\section{Preliminaries}
We write vectors with bold font, e.g. $\veca a=(\vecb a(1),\veca a(2), \veca a(3))$ is a vector with three coordinates.
For every positive integer $k$, $[k]$ denotes the set $\{1,\ldots,k\}$.

All graphs in this paper are finite and simple (that is, loops and multiple edges are not allowed).
For every natural number $n$, let $\mathcal M_n$ denote the set of all simple graphs on $n$ vertices up to an
isomorphism. For a graph $G$, let $V(G)$ and $E(G)$, respectively denote  the set of the vertices and the edges of $G$.
The complement of $G$ is denoted by $G^\ast$.

The \emph{homomorphism density} of a graph $H$ in a graph $G$, denoted by $t(H;G)$, is the probability that a random map from the vertices of $H$
to the vertices of $G$ is a {\em graph homomorphism}, that is it maps every edge of $H$ to an edge of $G$. If $H\in \mathcal M_\ell$,
$G\in\mathcal M_n$, and $\ell\le n$, then $t_0(H;G)$ denotes the probability that a random {\bf injective} map from $V(H)$ to $V(G)$ is a graph
homomorphism, and $p(H,G)$ denotes the probability that a random set of $\ell$ vertices of $G$ induces a graph isomorphic to $H$. We have the
following chain rule (cf. \cite[Lemma 2.2]{MR2371204}):
\begin{equation} \label{eq:chain}
t_0(H;G) = \sum_{F\in \mathcal M_\ell}t_0(H;F)p(F,G),
\end{equation}
where $|V(H)|\leq \ell\leq |V(G)|$.

\begin{definition}
A graph $H$ is called \emph{common} if
\begin{equation}
\label{eq:common}
\liminf_{n\to\infty} \min_{G\in\mathcal M_n}(t(H;G) + t(H;G^\ast)) \ge 2^{1-|E(H)|}.
\end{equation}
\end{definition}
It is easy to see that as $n \rightarrow \infty$,  for a random graph $G$ on $n$ vertices, we have, with high probability,
$t(H;G) + t(H;G^\ast) = 2^{1-|E(H)|} \pm o(1)$. Thus, $H$ is common if the total number of copies of $H$ in every graph
and its complement asymptotically minimizes for random graphs. Note also that since $t(H;G)$ and $t_0(H;G)$ are asymptotically equal (again,
as $n \rightarrow \infty$), one could use $t_0(H;G)$ in place of $t(H;G)$ in \eqref{eq:common}, and this is what we will do in our proof.

\subsection{Flag algebras}
We assume certain familiarity with the theory of flag algebras from~\cite{MR2371204}. However, for the proof of the central Theorem~\ref{thm:main} only the most basic notions are required. Thus, instead of trying to duplicate definitions, we occasionally give pointers to relevant places in~\cite{MR2371204}.

In our application of the flag algebras calculus we work exclusively with the theory of simple graphs (cf.~\cite[\S 2]{MR2371204}). As in~\cite{MR2371204}, flags of type $\sigma$ and size $k$ are denoted by $\mathcal F^\sigma_k$. The flag algebra generated by all flags of type $\sigma$ is denoted by $\mathcal A^\sigma$ (cf.~\cite[\S2]{MR2371204}). Apart from already defined model $W_5\in\mathcal M_6$ we need to introduce the following models, types, and flags.

We shall work with five types $\sigma_0,\sigma_1,\ldots,\sigma_4$ of size four which are illustrated in Figure \ref{types}.
\begin{figure}[tbp]
\begin{center}
% This file was made with: LaTeXPiX  (Build 3618)
% Coded by: N.J.H.M. van Beurden
% Email: beurden@email.com
% Webpage: http://www.beurden.cjb.net
\setlength{\unitlength}{0.254mm}
\begin{picture}(372,222)(30,-356)
        \special{color rgb 0 0 0}\allinethickness{0.254mm}\special{sh 0.99}\put(40,-145){\ellipse{4}{4}} % Shade Dot
        \special{color rgb 0 0 0}\allinethickness{0.254mm}\special{sh 0.99}\put(40,-195){\ellipse{4}{4}} % Shade Dot
        \special{color rgb 0 0 0}\allinethickness{0.254mm}\special{sh 0.99}\put(90,-145){\ellipse{4}{4}} % Shade Dot
        \special{color rgb 0 0 0}\allinethickness{0.254mm}\special{sh 0.99}\put(90,-195){\ellipse{4}{4}} % Shade Dot
        \special{color rgb 0 0 0}\put(60,-211){\shortstack{$\sigma_0$}} % Plain Text
        \special{color rgb 0 0 0}\allinethickness{0.254mm}\special{sh 0.99}\put(215,-145){\ellipse{4}{4}} % Shade Dot
        \special{color rgb 0 0 0}\allinethickness{0.254mm}\special{sh 0.99}\put(215,-195){\ellipse{4}{4}} % Shade Dot
        \special{color rgb 0 0 0}\allinethickness{0.254mm}\special{sh 0.99}\put(165,-145){\ellipse{4}{4}} % Shade Dot
        \special{color rgb 0 0 0}\allinethickness{0.254mm}\special{sh 0.99}\put(165,-195){\ellipse{4}{4}} % Shade Dot
        \special{color rgb 0 0 0}\allinethickness{0.254mm}\special{sh 0.99}\put(215,-145){\ellipse{4}{4}} % Shade Dot
        \special{color rgb 0 0 0}\allinethickness{0.254mm}\special{sh 0.99}\put(215,-195){\ellipse{4}{4}} % Shade Dot
        \special{color rgb 0 0 0}\put(185,-211){\shortstack{$\sigma_1$}} % Plain Text
        \special{color rgb 0 0 0}\allinethickness{0.254mm}\special{sh 0.99}\put(290,-145){\ellipse{4}{4}} % Shade Dot
        \special{color rgb 0 0 0}\allinethickness{0.254mm}\special{sh 0.99}\put(290,-195){\ellipse{4}{4}} % Shade Dot
        \special{color rgb 0 0 0}\allinethickness{0.254mm}\special{sh 0.99}\put(340,-145){\ellipse{4}{4}} % Shade Dot
        \special{color rgb 0 0 0}\allinethickness{0.254mm}\special{sh 0.99}\put(340,-195){\ellipse{4}{4}} % Shade Dot
        \special{color rgb 0 0 0}\put(310,-211){\shortstack{$\sigma_2$}} % Plain Text
        \special{color rgb 0 0 0}\allinethickness{0.254mm}\special{sh 0.99}\put(105,-290){\ellipse{4}{4}} % Shade Dot
        \special{color rgb 0 0 0}\allinethickness{0.254mm}\special{sh 0.99}\put(105,-340){\ellipse{4}{4}} % Shade Dot
        \special{color rgb 0 0 0}\allinethickness{0.254mm}\special{sh 0.99}\put(155,-290){\ellipse{4}{4}} % Shade Dot
        \special{color rgb 0 0 0}\allinethickness{0.254mm}\special{sh 0.99}\put(155,-340){\ellipse{4}{4}} % Shade Dot
        \special{color rgb 0 0 0}\put(125,-356){\shortstack{$\sigma_3$}} % Plain Text
        \special{color rgb 0 0 0}\allinethickness{0.254mm}\special{sh 0.99}\put(280,-290){\ellipse{4}{4}} % Shade Dot
        \special{color rgb 0 0 0}\allinethickness{0.254mm}\special{sh 0.99}\put(280,-340){\ellipse{4}{4}} % Shade Dot
        \special{color rgb 0 0 0}\allinethickness{0.254mm}\special{sh 0.99}\put(230,-290){\ellipse{4}{4}} % Shade Dot
        \special{color rgb 0 0 0}\allinethickness{0.254mm}\special{sh 0.99}\put(230,-340){\ellipse{4}{4}} % Shade Dot
        \special{color rgb 0 0 0}\allinethickness{0.254mm}\special{sh 0.99}\put(280,-290){\ellipse{4}{4}} % Shade Dot
        \special{color rgb 0 0 0}\allinethickness{0.254mm}\special{sh 0.99}\put(280,-340){\ellipse{4}{4}} % Shade Dot
        \special{color rgb 0 0 0}\put(250,-356){\shortstack{$\sigma_4$}} % Plain Text
        \special{color rgb 0 0 0}\allinethickness{0.254mm}\path(105,-340)(105,-290) % Plain Solid Line
        \special{color rgb 0 0 0}\allinethickness{0.254mm}\path(105,-340)(155,-340) % Plain Solid Line
        \special{color rgb 0 0 0}\allinethickness{0.254mm}\path(280,-340)(230,-340) % Plain Solid Line
        \special{color rgb 0 0 0}\allinethickness{0.254mm}\path(290,-145)(290,-195) % Plain Solid Line
        \special{color rgb 0 0 0}\allinethickness{0.254mm}\path(290,-195)(340,-195) % Plain Solid Line
        \special{color rgb 0 0 0}\put(155,-196){\shortstack{\scriptsize 1}} % Plain Text
        \special{color rgb 0 0 0}\put(220,-196){\shortstack{\scriptsize 2}} % Plain Text
        \special{color rgb 0 0 0}\put(155,-146){\shortstack{\scriptsize 3}} % Plain Text
        \special{color rgb 0 0 0}\put(220,-146){\shortstack{\scriptsize 4}} % Plain Text
        \special{color rgb 0 0 0}\allinethickness{0.254mm}\special{sh 0.99}\put(215,-145){\ellipse{4}{4}} % Shade Dot
        \special{color rgb 0 0 0}\allinethickness{0.254mm}\special{sh 0.99}\put(215,-195){\ellipse{4}{4}} % Shade Dot
        \special{color rgb 0 0 0}\allinethickness{0.254mm}\special{sh 0.99}\put(165,-145){\ellipse{4}{4}} % Shade Dot
        \special{color rgb 0 0 0}\allinethickness{0.254mm}\special{sh 0.99}\put(165,-195){\ellipse{4}{4}} % Shade Dot
        \special{color rgb 0 0 0}\allinethickness{0.254mm}\special{sh 0.99}\put(215,-145){\ellipse{4}{4}} % Shade Dot
        \special{color rgb 0 0 0}\allinethickness{0.254mm}\special{sh 0.99}\put(215,-195){\ellipse{4}{4}} % Shade Dot
        \special{color rgb 0 0 0}\put(155,-196){\shortstack{\scriptsize 1}} % Plain Text
        \special{color rgb 0 0 0}\put(220,-196){\shortstack{\scriptsize 2}} % Plain Text
        \special{color rgb 0 0 0}\put(155,-146){\shortstack{\scriptsize 3}} % Plain Text
        \special{color rgb 0 0 0}\put(220,-146){\shortstack{\scriptsize 4}} % Plain Text
        \special{color rgb 0 0 0}\put(30,-196){\shortstack{\scriptsize 1}} % Plain Text
        \special{color rgb 0 0 0}\put(95,-196){\shortstack{\scriptsize 2}} % Plain Text
        \special{color rgb 0 0 0}\put(30,-146){\shortstack{\scriptsize 3}} % Plain Text
        \special{color rgb 0 0 0}\put(95,-146){\shortstack{\scriptsize 4}} % Plain Text
        \special{color rgb 0 0 0}\put(280,-196){\shortstack{\scriptsize 1}} % Plain Text
        \special{color rgb 0 0 0}\put(345,-196){\shortstack{\scriptsize 2}} % Plain Text
        \special{color rgb 0 0 0}\put(280,-146){\shortstack{\scriptsize 3}} % Plain Text
        \special{color rgb 0 0 0}\put(345,-146){\shortstack{\scriptsize 4}} % Plain Text
        \special{color rgb 0 0 0}\put(220,-341){\shortstack{\scriptsize 1}} % Plain Text
        \special{color rgb 0 0 0}\put(285,-341){\shortstack{\scriptsize 2}} % Plain Text
        \special{color rgb 0 0 0}\put(220,-291){\shortstack{\scriptsize 3}} % Plain Text
        \special{color rgb 0 0 0}\put(285,-291){\shortstack{\scriptsize 4}} % Plain Text
        \special{color rgb 0 0 0}\put(95,-341){\shortstack{\scriptsize 1}} % Plain Text
        \special{color rgb 0 0 0}\put(160,-341){\shortstack{\scriptsize 2}} % Plain Text
        \special{color rgb 0 0 0}\put(95,-291){\shortstack{\scriptsize 3}} % Plain Text
        \special{color rgb 0 0 0}\put(160,-291){\shortstack{\scriptsize 4}} % Plain Text
        \special{color rgb 0 0 0}\allinethickness{0.254mm}\path(165,-195)(215,-195) % Plain Solid Line
        \special{color rgb 0 0 0}\allinethickness{0.254mm}\path(105,-340)(155,-290) % Plain Solid Line
        \special{color rgb 0 0 0}\allinethickness{0.254mm}\path(230,-290)(280,-290) % Plain Solid Line
        \special{color rgb 0 0 0} % Set color to black again (default font color)
\end{picture}
\caption{\label{types} Types.}
\end{center}
\end{figure}
For a type $\sigma$ of size $k$ and a set of vertices $V\subseteq [k]$ in $\sigma$, let $F^\sigma_V$ denote the flag $(G,\theta)\in \mathcal F^\sigma_{k+1}$ in which the only unlabeled vertex $v$ is connected to the set $\{\theta(i)\::\:i\in V\}$.
We further define $f^\sigma_V\in \mathcal A^\sigma$ by
$$
f^\sigma_V \df F^\sigma_\emptyset - \frac{1}{|\text{Aut}(\sigma)|}\cdot\sum_{\eta\in \text{Aut}({\sigma})} F^\sigma_{\eta(V)}.
$$
These elements form a basis (for $V\neq\emptyset$ and with repetitions) in the space spanned by those $f\in\mathcal A^\sigma_{k+1}$ that are both
$\text{Aut}(\sigma)$-invariant and asymptotically vanish on random graphs; other than that, our particular choice of elements with this
property is more or less arbitrary.

Recall that in~\cite[\S2.2]{MR2371204} a certain ``averaging operator'' $\eval{\cdot}{}$ was introduced. This operator plays a central role in the flag algebra calculus.

Let $\ast \in \text{Aut}(\mathcal A^0)$ be the involution that corresponds to taking the complementary graph. That is, we  extend $\ast$ linearly from $\bigcup_n \mathcal M_n$ to $\mathcal A^0$.

\section{Main result}
We can now state the main result of the paper.
\begin{theorem}
\label{thm:main}
The $5$-wheel $W_5$ is common.
\end{theorem}
\begin{proof}
Let $\widehat W_5\in\mathcal A^0$ be the element that counts the injective homomorphism density of the 5-wheel, that is
$$
\widehat W_5\df \sum_{F\in \mathcal M_6} t_0(W_5, F)F.
$$
We shall prove that
\begin{equation}\label{eq:FAformulation}
\widehat W_5+\widehat W_5^\ast \geq 2^{-9}\;,
\end{equation}
where the inequality $\leq$ in the algebra $\mathcal A^0$ is defined in \cite[Definition 6]{MR2371204}. An alternate interpretation of this
inequality \cite[Corollary 3.4]{MR2371204} is that
$$
\liminf_{n\to\infty} \min_{G\in \mathcal M_n} (p(\widehat W_5,G)+p(\widehat W_5^\ast,G)) \geq 2^{-9}.
$$
Since $p(\widehat W_5,G)=\sum_{F\in \mathcal M_6}t_0(\widehat W_5; F)p(F;G)=t_0(\widehat W_5; G)$ by \eqref{eq:chain}, and, likewise,
$p(\widehat W_5^\ast, G) = p(\widehat W_5, G^\ast)= t_0(\widehat W_5; G^\ast)$, \eqref{eq:FAformulation} implies Theorem \ref{thm:main}.

\bigskip
We now give a proof of~\eqref{eq:FAformulation}. To this end we work with suitable quadratic forms $Q_{\sigma_i}^{+/-}$ defined by symmetric
matrices $M_{\sigma_i}^{+/-}$ and vectors $\veca g_i^{+/-}$ in the algebras $\mathcal A^{\sigma_i}$. The numerical values of the matrices
$M_{\sigma_i}^{+/-}$ and vectors $\veca g_i^{+/-}$ are given in the appendix. It is essential that all the matrices $M_{\sigma_i}^{+/-}$ are
positive definite which can be verified using any general mathematical software. Next we define
$$R:=\left(\sum_{i=0}^4 \eval{Q_{\sigma_i}^+(\veca g_i^+)}{\sigma_i}\right) + \eval{Q_{\sigma_1}^-(\veca g_1^-)}{\sigma_1} + \eval{Q_{\sigma_4}^-(\veca g_4^-)}{\sigma_4}\;.$$
We claim that
\begin{equation} \label{eq:equality}
\widehat W_5+\widehat W_5^\ast = 2^{-9} +R+R^\ast.
\end{equation}
All the terms in~\eqref{eq:equality} can be expressed as linear combinations of graphs from $\mathcal M_6$ and thus checking~\eqref{eq:equality} amounts to checking the coefficients of the 156 flags from $\mathcal M_6$. We offer a \texttt{C}-code available at \texttt{http://kam.mff.cuni.cz/$\sim$kral/wheel} that verifies the equality~\eqref{eq:equality}.

By \cite[Theorem~3.14]{MR2371204}, we have $$\left(\sum_{i=0}^4 \eval{Q_{\sigma_i}^+(\veca g_i^+)}{\sigma_i}\right) + \eval{Q_{\sigma_1}^-(\veca g_1^-)}{\sigma_1} + \eval{Q_{\sigma_4}^-(\veca g_4^-)}{\sigma_4}\ge 0\;.$$
Therefore,~\eqref{eq:equality} implies~\eqref{eq:FAformulation}.
\end{proof}

Theorem~\ref{thm:main} shows that a typical random graph $G=G_{n,\frac12}$ asymptotically minimizes the quantity $t(W_5;G) + t(W_5;G^{\ast})$.
Extending our method, we convinced ourselves
that $G_{n,\frac12}$ is essentially the only minimizer of $t(W_5;G) + t(W_5;G^{\ast})$. In terms of flag algebras this means that
the homomorphism $\phi \in \Hom^+(\mathcal{A}^0, {\Bbb R})$ (see \cite[Definition 5]{MR2371204}) satisfying
$\phi(\widehat  W_5+\widehat W_5^\ast)=2^{-9}$ is unique.

The outline of the argument is as follows. Let $\rho \in \mathcal M_2$ denote a graph consisting of a single edge, let  $C_4 \in \mathcal{M}_4$ denote the cycle of length $4$, and,
as before, let $$
\widehat C_4 \df \sum_{F\in \mathcal M_4} t_0(C_4; F)F.$$ The Erd\H{o}s-Simonovits-Sidorenko conjecture is known for $C_4$~\cite{MR1138091}, and
it implies that $\widehat C_4 \geq \rho^4$ and $\widehat C_4^\ast\geq(1-\rho)^4$ in $\mathcal A^0$. Therefore, $C_4+C_4^\ast\geq 1/8$ (i.e., $C_4$
is common), and, moreover, every $\phi \in \Hom^+(\mathcal{A}^0, {\Bbb R})$ attaining equality must satisfy $\phi(\rho)=1/2$ and $\phi(\widehat C_4)=1/16$.

On the other hand, is is shown in
~\cite{ChGrWi} that the density of edges and the density cycles of length $4$ characterize quasi-random graphs, implying that the homomorphism $\phi$ satisfying $\phi(\widehat C_4 + \widehat C_4^\ast)=1/8$ is
unique (and corresponds to quasi-random graphs).
Therefore, to verify the uniqueness of the homomorphism  $\phi$  satisfying $\phi(\widehat  W_5+\widehat W_5^\ast)=2^{-9}$ it suffices to show that
\begin{equation}\label{eq:uniqueness}
\widehat W_5+\widehat W_5^\ast \geq 2^{-9} + \frac{1}{100}\left( \widehat C_4+\widehat C_4^\ast -1/8\right).
\end{equation}
We have used a computer program to verify (\ref{eq:uniqueness}), and it is telling us that this inequality holds with quite a convincing
level of accuracy $10^{-10}$. But we have not converted the floating point computations into a rigorous proof.

\section{Conclusion}
In this paper we have exhibited the first example of a common graph that is not three-colorable. This naturally gives rise to
the following interesting question: do there exist common graphs with arbitrarily large chromatic number?

\bibliographystyle{alpha}
\bibliography{wheel}

\appendix
\section{The matrices $M^{+/-}_i$ and the vectors $\veca g_i^{+/-}$}
Here, we list the numerical values of the matrices $M^{+/-}_i$ and the vectors $\veca g_i^{+/-}$.
% These values were obtained with the help of a computer. While checking~\eqref{eq:equality} is in theory possible by hand, the reader may find useful our Maple verification of these calculations, available at $\texttt{http://??}$.
%  The file also contains a proof that the matrices $M^{+/-}_i$ are indeed positive semidefinite.

\smallskip
The vectors $\veca g_i^+$ are given by the tuples
\begin{eqnarray*}
\veca g_0^+ &\df& (f^{\sigma_0}_{\{1\}}, f^{\sigma_0}_{\{1,2\}}, f^{\sigma_0}_{\{1,2,3\}}, f^{\sigma_0}_{\{1,2,3,4\}})\\
\veca g_1^+ &\df& (f^{\sigma_1}_{\{1\}}, f^{\sigma_1}_{\{3\}}, f^{\sigma_1}_{\{1,2\}}, f^{\sigma_1}_{\{1,3\}}, f^{\sigma_1}_{\{3,4\}}, f^{\sigma_1}_{\{1,2,3\}}, f^{\sigma_1}_{\{1,3,4\}}, f^{\sigma_1}_{\{1,2,3,4\}})\\
\veca g_2^+ &\df& (f^{\sigma_2}_{\{1\}}, f^{\sigma_2}_{\{2\}}, f^{\sigma_2}_{\{4\}}, f^{\sigma_2}_{\{1,2\}}, f^{\sigma_2}_{\{1,4\}}, f^{\sigma_2}_{\{2,3\}}, f^{\sigma_2}_{\{2,4\}}, f^{\sigma_2}_{\{1,2,3\}}, f^{\sigma_2}_{\{1,2,4\}}, f^{\sigma_2}_{\{2,3,4\}}, f^{\sigma_2}_{\{1,2,3,4\}})\\
\veca g_3^+ &\df& (f^{\sigma_3}_{\{1\}}, f^{\sigma_3}_{\{2\}},
f^{\sigma_3}_{\{1,2\}}, f^{\sigma_3}_{\{2,3\}}, f^{\sigma_3}_{\{1,2,3\}}, f^{\sigma_3}_{\{2,3,4\}}, f^{\sigma_3}_{\{1,2,3,4\}})\\
\veca g_4^+ &\df& (f^{\sigma_4}_{\{1\}}, f^{\sigma_4}_{\{1,2\}}, f^{\sigma_4}_{\{1,3\}}, f^{\sigma_4}_{\{1,2,3\}}, f^{\sigma_4}_{\{1,2,3,4\}}),
\end{eqnarray*}
and the vectors $\veca g_i^-$ are given by
\begin{eqnarray*}
\veca g_1^- &\df& (F^{\sigma_1}_{\{3\}}-F^{\sigma_1}_{\{4\}}, F^{\sigma_1}_{\{1,3,4\}}-F^{\sigma_1}_{\{2,3,4\}}, F^{\sigma_1}_{\{1,3\}}-F^{\sigma_1}_{\{2,3\}}, F^{\sigma_1}_{\{1,3\}}-F^{\sigma_1}_{\{2,4\}}, F^{\sigma_1}_{\{1,3\}}-F^{\sigma_1}_{\{3,4\}})\\
\veca g_4^- &\df& (F^{\sigma_4}_{\{1,2\}}-F^{\sigma_4}_{\{3,4\}}, F^{\sigma_4}_{\{1,3\}}-F^{\sigma_4}_{\{2,3\}}, F^{\sigma_4}_{\{1,3\}}-F^{\sigma_4}_{\{2,4\}}, F^{\sigma_4}_{\{1,3\}}-F^{\sigma_4}_{\{3,4\}}).
\end{eqnarray*}
The matrices $M^{+/-}_i$ are listed on the next two pages.

\begin{align*}
M_0^+&\df \frac 1{2\cdot 10^8}\times\left( \begin {array}{cccc} 104133330&-67645847&-126443014&-53041562
\\\noalign{\medskip}-67645847&58559244&68999274&28961030
\\\noalign{\medskip}-126443014&68999274&166581934&69653308
\\\noalign{\medskip}-53041562&28961030&69653308&29368489\end {array}
 \right)
\\
M_1^+&\df \frac 1{24\cdot 10^8} \times\\
\times&{\tiny\left( \begin {array}{cccccccc} 3376427096&-550659377&1175122309&-
274818336&-1951510989&133242698&-2978772360&-1118255328
\\\noalign{\medskip}-550659377&3579306230&-2818779263&254758382&
1853810147&-3593215008&1149060744&-2243131164\\\noalign{\medskip}
1175122309&-2818779263&2446135762&-153160723&-1883990616&2571244464&-
1644918408&1392930672\\\noalign{\medskip}-274818336&254758382&-
153160723&259013952&207245488&-524428416&59129384&-87439632
\\\noalign{\medskip}-1951510989&1853810147&-1883990616&207245488&
2026568566&-1339529064&2075124696&-196178016\\\noalign{\medskip}
133242698&-3593215008&2571244464&-524428416&-1339529064&4383894552&-
474279456&2753404296\\\noalign{\medskip}-2978772360&1149060744&-
1644918408&59129384&2075124696&-474279456&2987175794&578705400
\\\noalign{\medskip}-1118255328&-2243131164&1392930672&-87439632&-
196178016&2753404296&578705400&2302497768\end {array} \right)}
\\
M_3^+&\df \frac 1{24\cdot 10^8}\times \\ \times& {\tiny\left( \begin {array}{ccccccc} 1770465360&-40788068&770354664&-
280179622&-1109635560&-593033461&-1434435065\\\noalign{\medskip}-
40788068&503182008&-377074674&-65682192&-316936632&337167432&-
405260664\\\noalign{\medskip}770354664&-377074674&942288720&-5442408&-
584215338&-635915808&-299584920\\\noalign{\medskip}-280179622&-
65682192&-5442408&90869472&187091280&-48623352&356458176
\\\noalign{\medskip}-1109635560&-316936632&-584215338&187091280&
1325422128&196268064&1280101992\\\noalign{\medskip}-593033461&
337167432&-635915808&-48623352&196268064&706802676&-31363774
\\\noalign{\medskip}-1434435065&-405260664&-299584920&356458176&
1280101992&-31363774&1763018404\end {array} \right)}
\\
M_4^+&\df \frac 1{12\cdot 10^8} \times\left( \begin {array}{ccccc} 6589068&-137160&60408&-3635796&-5354976
\\\noalign{\medskip}-137160&3975070&-399180&-720636&-1388043
\\\noalign{\medskip}60408&-399180&3506988&-1778640&-3413616
\\\noalign{\medskip}-3635796&-720636&-1778640&5107716&3969708
\\\noalign{\medskip}-5354976&-1388043&-3413616&3969708&12276592
\end {array} \right)
%\end{align*}
%\begin{align*}
\\
M_1^-&\df \frac 1{48\cdot 10^8} \times\left( \begin {array}{ccccc} 1871684759&828164352&153135600&
2205677647&32494800\\\noalign{\medskip}828164352&647325323&122226960&
1702274830&23569680\\\noalign{\medskip}153135600&122226960&32894794&
317036160&988560\\\noalign{\medskip}2205677647&1702274830&317036160&
4533494520&62236800\\\noalign{\medskip}32494800&23569680&988560&
62236800&7445060\end {array} \right)
\\
M_4^-&\df \frac 1{24\cdot 10^8} \times\left( \begin {array}{cccc} 371929992&-665160&31885344&6896381
\\\noalign{\medskip}-665160&4952616&15347271&-425892
\\\noalign{\medskip}31885344&15347271&420643536&5244336
\\\noalign{\medskip}6896381&-425892&5244336&1704738\end {array}
 \right).
\end{align*}
\begin{landscape}
\begin{align*}
M_2^+&\df \frac 1{24\cdot 10^8} \times \\
\times&\tiny{\left(\begin{array}{ccccccccccc} 4114457904&-2123660510&578302533&
2402100408&1609339896&-4979381511&-1073916061&-711542544&-108075291&-
311854200&-1172726832\\\noalign{\medskip}-2123660510&4697332052&-
146727648&-2893487330&-831349224&5132020824&1140828192&-2533278088&-
3120849612&586989168&-2130186959\\\noalign{\medskip}578302533&-
146727648&2842930424&-2377739616&2453284752&-1134538157&949692648&-
2122945241&799767696&-646840455&-1452441435\\\noalign{\medskip}
2402100408&-2893487330&-2377739616&5029589784&-1305679056&-3694198620&
-1628657160&2987352093&-17138568&174993936&1346820763
\\\noalign{\medskip}1609339896&-831349224&2453284752&-1305679056&
2899169976&-3008866416&227603736&-2158976640&1272333144&-824389152&-
1468496784\\\noalign{\medskip}-4979381511&5132020824&-1134538157&-
3694198620&-3008866416&9045922946&1585531176&-492543642&-2720802624&
1167719184&119548200\\\noalign{\medskip}-1073916061&1140828192&
949692648&-1628657160&227603736&1585531176&1198933584&-594013398&-
787158072&-14360286&-864511462\\\noalign{\medskip}-711542544&-
2533278088&-2122945241&2987352093&-2158976640&-492543642&-594013398&
4445640792&1152146526&408353664&3139157376\\\noalign{\medskip}-
108075291&-3120849612&799767696&-17138568&1272333144&-2720802624&-
787158072&1152146526&4353119928&-778415544&2410765872
\\\noalign{\medskip}-311854200&586989168&-646840455&174993936&-
824389152&1167719184&-14360286&408353664&-778415544&430490652&
217228440\\\noalign{\medskip}-1172726832&-2130186959&-1452441435&
1346820763&-1468496784&119548200&-864511462&3139157376&2410765872&
217228440&3407087808\end {array} \right)}
\end{align*}
\end{landscape}
\end{document}